\documentclass{lms}
\usepackage{amsfonts}
\usepackage{amsmath}
\usepackage{amssymb}
\usepackage{mathrsfs}

\newtheorem{theorem}{Theorem}[section] 
\newtheorem{lemma}[theorem]{Lemma}     

\usepackage{natbib}
\setcitestyle{authoryear,open={(},close={)}} 

\newnumbered{assertion}{Assertion}    
\newnumbered{conjecture}{Conjecture}  
\newnumbered{definition}{Definition}
\newnumbered{hypothesis}{Hypothesis}
\newnumbered{remark}{Remark}
\newnumbered{note}{Note}
\newnumbered{observation}{Observation}
\newnumbered{problem}{Problem}
\newnumbered{question}{Question}
\newnumbered{algorithm}{Algorithm}
\newnumbered{example}{Example}
\newunnumbered{notation}{Notation} 

\title[Martingale Problem and Quadratic Family]
 {Martingale Problem and Quadratic Family} 

\author{Haoming Wang}

\classno{60G46 (primary), 60J10, 60J28, 60J35 (secondary).}

\extraline{\emph{Keywords}: Martingale problem, Girsanov transformation, Markov process, Quadratic family}

\begin{document}
\maketitle

\begin{abstract}Assuming uniqueness of the martingale problem for Markov processes of generators $q_t$ in a quadratic family like
\[q_t(i,j) = a_t(i) q_0(i,j)^2 + b_t(i) q_0(i,j) - \frac{a_t(i)}{N} \sum_k q_0(i,k)^2,\]
where $a_t(i),b_t(i)$ are predictable processes, $N$ is the number of states, and $q_0$ represents the generator of a stationary reference Markov process which satisfies $q_0(i,j)>0$ for all $i,j$, we obtain the sufficient and necessary conditions for the Girsanov transformation.
\end{abstract}

\section{Introduction}

A natural question in the theory of Markov processes is to recover a probability measure $\mathbf{P}$ from the martingale problem. Specifically, fix a finite state space $\{1,2,\dots,N\}$, an initial distribution $\nu = (\nu_1,\dots,\nu_N)$ with $\nu_i \geq 0$ and $\sum_{i=1}^N \nu_i = 1$, and a sequence of transition matrices $P_n$. The problem is to construct a probability measure $\mathbf{P}$ on the path space such that, for every $z \in \mathbb{R}^N$, the process
$$
z_{X_n} - z_{X_0} - \sum_{k=1}^n (P_k z)_{X_{k-1}}
$$
is a $\mathbf{P}$-martingale. Once such a measure exists, the process $(X_n)$ becomes an inhomogeneous Markov chain with transition matrices $P_n$. A continuous-time analogue of this formulation will be presented in the sequel.

This problem has a unique solution for the discrete-time Markov chain by recursive construction of finite-dimensional distributions. However, as had been pointed out by \cite{StroockVaradhan1969}, solutions to the martingale problem for the continuous-time Markov process are generally not unique. In this paper, we aim to prove 
such a result which will be briefly stated as follows. Let $\mathcal S$ be a locally compact $T_2$ space equipped with Borel measure $\mu$, and $\nu$ be an initial distribution satisfying $\nu(x)\geq 0$ and $\int_{\mathcal S}\nu(x)\mu(dx)=1$. For each compactly supported test function $f \in C_c(\mathcal S)$, suppose there exists a unique probability measure $\mathbf P$ such that
$$f(X_t) - f(X_0) - \int_0^t (q_s f)(X_s)\,ds$$
is a $\mathbf P$-martingale, where the generator $q_t$ belongs to the quadratic family
$$q_t(i,j) \;=\; a_t(i) q_0(i,j)^2 + b_t(i) q_0(i,j) - \frac{a_t(i)}{N}\sum_k q_0(i,k)^2,  \qquad a_t(i), b_t(i) \in \mathcal P,$$
with $\mathcal P$ denoting the predictable $\sigma$-algebra. Then the following conclusion holds. For a stationary Markov process $X_t$ with generator $q_{0}$, $q_0(i,j)>0$, $\mu$-a.e. under $\mathbf P_0$ and $\mathbf P \ll \mathbf P_0$, the Girsanov transformation (*) holds with Dol\'eans-Dade exponential $\mathcal E_t(\cdot)$,
\[\frac{d\mathbf P}{d\mathbf P_0}\Big|_{\mathscr F_t}= \mathcal E_t\left(\int_0^t \int_{\mathcal S} 
\log\Big(\frac{q_s(X_{s-},y)}{q_0(X_{s-},y)}\Big)  \big(N(ds,dy)-q_0(X_{s-},dy)ds\big)\right), \tag{*}\]
if and only if $X_t$ is a Markov process with generator $q_{t}$ under $\mathbf P$. In (*), $\mathscr F_t = \sigma(X_s:0\le s \le t)$ is the smallest $\sigma$-algebra generated by $X_s$ up to time $t$, and $N(dt,dy)$ is the jump measure of $X_t$ defined by
$$N(\omega; [0,t] \times A) := \sum_{0 < s \le t} \mathbf 1{\{X_s(\omega) \neq X_{s-}(\omega), X_s(\omega) \in A\}}, \quad A \in \mathscr B(\mathcal S),$$
where $X_{s-}$ denotes the left limit of $X$ at $s$.

The proof will be given in the next two sections, first for the discrete time. We need to mention that the main technique we adopt is not new. Instead, the martingale representation theorem is well-known for Markov processes and can be seen in many expositions like \cite{clark1970representation}, \cite{buiculescu1982representation}, \cite{Sokol03092015}, and \cite{Criens2024OnTR}. The transformation of measures for the Markov process was systematically developed by \cite{girsanov1960}, \cite{dynkin1965markov}, \cite{DoleansDade1970}, \cite{novikov1972identity}, and \cite{kazamaki1994continuous}. Here, for the Markov process on the finite state space, this formula (*) reads
\[\left.\frac{d\mathbf P}{d\mathbf P_0}\right|_{\mathscr F_t} = \exp\left( \int_0^t \sum_{j \neq X_s} \log \frac{q_{s}(X_s, j)}{q_{0}(X_s, j)} \, dN_s^j - \int_0^t \sum_{j \neq X_s} \left( q_s(X_s, j) - q_{0}(X_s, j) \right) ds \right), \tag{C} \]
where $N_s^j$ is the number of jumps into state $j$ by time $s$. Without loss of generality, we will restrict our attention to an enumerable set $\mathcal S$ but focus more on the time index set $\mathbb N$ or $[0,\infty)$.

\section{Main Theorem: Discrete time}

\begin{theorem}
    Let $X_{n}$ be a Markov chain with $\sigma(X_{0},X_{1},\dots, X_{n})$-predictable transition probability matrix $P_n$ under $\mathbf P$. Suppose $\mathbf{P}_{0}$ is a probability measure with respect to which $X_{n}$ is the stationary Markov chain with transition probability matrix $P_0$ satisfying $P_{0}(i.j) >0$ and the same initial distribution, then $\mathbf{P} \ll \mathbf{P}_{0}$ and the Girsanove transformation (D) holds,
    \begin{equation*}   \left.d\mathbf{P}/d\mathbf{P}_{0}\right|_{\sigma(X_{0},X_{1},\dots, X_{n})}=  
    \exp\left(\sum_{k=1}^{n} \log \frac{P_{k}(X_{k-1},X_{k})}{P_{0}(X_{k-1},X_{k})}
    \right). \tag{D}    \end{equation*}
    Conversely, if there exists $\mathbf{P}_{0}$ as described above and $\mathbf{P}\ll \mathbf{P}_{0}$, then there exists a $\sigma(X_{0},X_{1},\dots,X_{n})$-predictable transition probability matrix $P_n$ such that $X_n$ is a Markov chain with transition probability matrix $P_n$ under $\mathbf P$ and (D) holds. 
\end{theorem}
\begin{proof} Sufficiency. By uniqueness of the martingale problem for discrete-time Markov chains, it suffices to prove $\mathbf P \ll \mathbf P_0$ if $X_n$ has the Markov property under both $\mathbf P$ and $\mathbf P_0$, and is stationary under $\mathbf P_0$.  Let 
\[Z_n = \exp\left(\sum_{k=1}^{n} \log \frac{P_{k}(X_{k-1},X_{k})}{P_{0}(X_{k-1},X_{k})}
    \right),\]
Then $Z_n$ is a $\mathbf P_0$-martingale by the stationarity of $X_n$ under $\mathbf P_0$, which can be seen via the followiong calculation
$$\mathbf E_{0}\left[\frac{P_{n}(X_{n-1},X_{n})}{P_{0}(X_{n-1},X_{n})}\Bigm|\mathscr F_{n-1}\right]   =\sum_{i=1}^N\frac{P_{n}(X_{n-1},i)}{P_{0}(X_{n-1},i)}P_0(X_{n-1},i)=1.   $$
Define a probability measure $\mathbf P'$ on $\sigma(X_0,\dots,X_n)$ by $d\mathbf P' = Z_n \, d\mathbf P_0$. Intuitively, $\mathbf P$ and $\mathbf P'$ coincide on $\sigma(X_{0},\dots,X_{n})$. Since $Z_n$ is bounded and hence uniformly integrable, this identification extends to the full $\sigma$-field $\sigma\big(\bigcup_n \sigma(X_{0},\dots,X_{n})\big)$. Thus $\mathbf P = \mathbf P'$, and $(D)$ holds.

    \begin{lemma}
        Let $(X_n, P_n)$ be an adapted process, where $P_n$ is an $\mathscr F_{n-1}$-measurable random transition probability matrix with row sum 1. In order that for every $z \in \mathbb R^N$ the process
        \begin{equation}
            M_n^z = z_{X_n} - z_{X_0} - \sum_{k=1}^{n} \big( (P_k - I)z\big)_{X_{k-1}}
            \label{eq: Markov martingale}
        \end{equation}
        is an  $\mathscr F_n$-martingale, it is sufficient and necessary that $X_n$ is a Markov chain with transition matrix $P_n$ with respect to $\mathscr F_n$. 
        \end{lemma}
    \begin{proof} Suppose $X_n$ is a Markov chain with transition matrices $P_n$ with respect to $\mathscr F_n$. Then for every $n$ and $j$,
    $$\mathbf P(X_{n} = j \mid \mathscr F_{n-1}) = P_{n, X_{n-1}, j}.$$
    Or equivalently, for every $z \in \mathbb R^N$,
    $$\mathbf E[z_{X_{n}} \mid \mathscr F_{n-1}] = (P_{n} z)_{X_{n-1}}.$$
    Now consider $M_n^z$. To show it is a martingale, we compute
    $$M_{n}^z = z_{X_{n}} - z_{X_0} - \sum_{k=1}^{n} \big( (P_k-I) z\big)_{X_{k-1}}.$$
    Thus,
    $$M_{n}^z - M_{n-1}^z = z_{X_{n}} - z_{X_{n-1}} -\big( (P_{n}-I) z\big)_{X_{n-1}} = z_{X_{n}} - (P_{n} z)_{X_{n-1}}.$$
    Taking conditional expectation,
    $$\mathbf E[M_{n}^z - M_{n-1}^z \mid \mathscr F_{n-1}] = \mathbf E[z_{X_{n}} \mid \mathscr F_{n-1}] - (P_{n} z)_{X_{n-1}} = 0.$$
    Hence, $\mathbf E[M_{n}^z \mid \mathscr F_{n-1}] = M_{n-1}^z$, so $M_n^z$ is a martingale.

    Conversely, suppose $M_n^z$ is a martingale for every $z \in \mathbb R^N$. Then
    $$\mathbf E[M_{n}^z - M_{n-1}^z \mid \mathscr F_{n-1}] = 0.$$
    But
    $$M_{n}^z - M_{n-1}^z = z_{X_{n}} - (P_{n} z)_{X_{n-1}},$$
    so
    $$\mathbf E[z_{X_{n}} \mid \mathscr F_{n-1}] = (P_{n} z)_{X_{n-1}}.$$
    This holds for all $z \in \mathbb R^N$. In particular, choose $z = (z_1,z_2,\dots,z_N)$ to be the indicator vector: $z_i = \mathbf 1{\{i=j\}}$. Then $z_{X_{n}} = \mathbf 1{\{X_{n} = j\}}$, and
    $$(P_{n} z)_{X_{n-1}} = P_{n, X_{n-1}, j}.$$
    Thus,
    $$\mathbf P(X_{n} = j \mid \mathscr F_{n-1}) = P_{n, X_{n-1}, j}.$$
    This is exactly the Markov property with transition matrices $P_n$ with respect to $\mathscr F_n$. \end{proof}

    Necessity, Since $\mathbf{P} \ll \mathbf{P}_0$, we define the likelihood ratio process
    $$Z_n = \mathbf E_0 \left[ \frac{d\mathbf{P}}{d\mathbf{P}_0} \Bigm| {\sigma(X_0, X_1, \dots, X_n)}\right].$$
    The process $Z_n$ is a $\mathbf{P}_0$-martingale, and $Z_n > 0$ $\mathbf{P}_0$-a.s. For a fixed \( n \ge 1 \), define for each \( j \),
    \[    \Delta_{n,j} := \mathbf{1}{\{X_n = j\}} - P_0(X_{n-1}, j).\]

    \begin{lemma}\label{lem: delta-basis} For every \( \mathscr{F}_n \)-measurable random variable \( Y_n \) satisfying \( \mathbf{E}_0[Y_n \mid \mathscr{F}_{n-1}] = 0 \), there exist \( \mathscr{F}_{n-1} \)-measurable coefficients \( G_{n,j} \) such that
    \[Y_n = \sum_{j=1}^N G_{n,j}\Delta_{n,j}, \qquad \mathbf{P}_0\text{-a.s.}.\]
    and on each atom \( \{X_{n-1} = i\} \), the coefficients \( G_{n,j} \) are uniquely determined by a linear system.
    \end{lemma}

\begin{proof}
Consider an atom \( A_i := \{X_{n-1} = i\} \) for some \( i\). On \( A_i \), define the random vector \( (Y_n(i,j))_{j=1}^N \) by setting \( Y_n(i,j) = Y_n \cdot \mathbf{1}{\{X_n = j\}} \) for each \( j \). Note that \( Y_n(i,j) \) represents the value of \( Y_n \) on the event \( \{X_n = j\} \cap A_i \). The condition \( \mathbf{E}_0[Y_n \mid \mathscr{F}_{n-1}] = 0 \) implies that
\[\sum_{j=1}^N  P_0(i,j) Y_n(i,j) = 0,\]
since the conditional expectation given \( \mathscr{F}_{n-1} \) reduces to averaging over \( X_n \) with weights \( P_0(i,j) \) on \( A_i \).

Now, consider the vectors \( \Delta_{n,j}(i, \cdot) \) for \(j\), where \( \Delta_{n,j}(i, k) = \mathbf{1}{\{k = j\}} - P_0(i,j) \) for each \(k\). These vectors span the subspace
\[V_i = \left\{ v \in \mathbb{R}^N : \sum_{j=1}^N  P_0(i,j) v_j = 0 \right\},\]
because any vector \( v \in V_i \) can be written as a linear combination of the vectors \( \Delta_{n,j}(i, \cdot) \). Specifically, the system of equations
\[Y_n(i, k) = \sum_{j=1}^N  G_{n,j}(i) \left[ \mathbf{1}{\{k = j\}} - P_0(i,j) \right], \qquad k = 1,2,\dots, N,\]
has a solution for coefficients \( G_{n,j}(i) \) since the right-hand side covers all vectors in \( V_i \). The solution is unique up to the constraints of the subspace.

Since this holds for each atom \( A_i \), we can define \( \mathscr{F}_{n-1} \)-measurable coefficients \( G_{n,j} \) by setting \( G_{n,j} = G_{n,j}(i) \) on \( A_i \). Then, the decomposition
\[Y_n = \sum_{j=1}^N G_{n,j} \, \Delta_{n,j}\]
holds \( \mathbf{P}_0 \)-almost surely, as required.
\end{proof}

For each $k \leq n$, set $Y_k = Z_k - Z_{k-1}$. Since $(Z_k)$ is a martingale, $\mathbf{E}_0[Y_k \mid \mathscr{F}_{k-1}] = 0$. By Lemma \ref{lem: delta-basis}, there exist $\mathscr{F}_{k-1}$-measurable coefficients $(G_{k,j})_{j=1}^N$ such that
$$Y_k = \sum_{j=1}^N G_{k,j}\,\Delta_{k,j}, 
\qquad \Delta_{k,j} = \mathbf{1}{\{X_k=j\}} - P_0(X_{k-1},j),$$
which holds $\mathbf P_0$-almost surely. Intuitively, the variables $\Delta_{k,j}$ represent centered indicators of the transition into state $j$, so the above identity is a martingale representation of $Y_k$. 

Next, we identify the transition probabilities $P_k(i,j)$ under $\mathbf P$. By the likelihood ratio property, we have under $\mathbf P_0$ the recursion
$$Z_k = Z_{k-1}\frac{P_{k}(X_{k-1},X_k)}{P_0(X_{k-1},X_k)},$$
so that comparing this expression with the martingale representation of $Y_k$ and restricting to the event $\{X_{k-1}=i,\,X_k=j\}$,
$$Z_{k-1}\!\left(\frac{P_{k}(i,j)}{P_0(i,j)}-1\right) = G_{k,j} - \sum_{l=1}^N G_{k,l}P_0(i,l).$$
Therefore, whenever $Z_{k-1}>0$, the transition probabilities under $\mathbf P$ are given explicitly by
$$P_{k}(i,j) = P_0(i,j)\left(1+\frac{1}{Z_{k-1}}\Big(G_{k,j}-\sum_{l=1}^N G_{k,l}P_0(i,l)\Big)\right).$$

It remains to check that these coefficients indeed form a stochastic matrix. Fixing $i$ and summing over $j$, we obtain
$$
\sum_{j=1}^N P_{k}(i,j)
= \sum_{j=1}^N P_0(i,j) 
+ \frac{1}{Z_{k-1}}\!\left(\sum_{j=1}^N P_0(i,j)G_{k,j} - \sum_{l=1}^N G_{k,l}P_0(i,l)\right).
$$
The second term vanishes because both expressions inside coincide, leaving
$$\sum_{j=1}^N P_0(i,j)  = 1.$$
Thus each row sums to one, and non-negativity holds by construction, since the $P_k(i,j)$ arise as conditional probabilities. Consequently, $P_k$ is a valid transition matrix adapted to the natural filtration. In particular, the process $X_n$ is a Markov chain under $\mathbf{P}$, and the Girsanov transform (D) follows.  This complete the proof.  \end{proof}

\section{Main Theorem: Continuous time}

\begin{theorem}
Let \(X_t\) be a continuous-time Markov process with $\sigma(X_s:0 \le s \le t)$-predictable generator \(q_t\) under \(\mathbf{P}\). Suppose \(\mathbf{P}_0\) is a probability measure with respect to which \(X_t\) is a stationary Markov process with generator \(q_0\) satisfying \(q_0(i,j) > 0\) for all \(i \neq j\), and the same initial distribution. Further assume the martingale problem has unique solutions for each $q_t$ in the quadratic family
\[q_t(i,j) = a_t(i) q_0(i,j)^2 + b_t(i) q_0(i,j) - \frac{a_t(i)}{N} \sum_k q_0(i,k)^2, \quad a_t(i),b_t(i) \in \mathcal P.\]
Then \(\mathbf{P} \ll \mathbf{P}_0\) and the Girsanov transformation (C) holds.

Conversely, if there exists \(\mathbf{P}_0\) as described above and \(\mathbf{P} \ll \mathbf{P}_0\), then there exists a $\sigma(X_s:0 \le s \le t)$-predictable generator \(q_t\) such that \(X_t\) is a Markov process with $\sigma(X_s:0 \le s \le t)$-predictable process \(q_t\) under \(\mathbf{P}\) and (*) holds.
\end{theorem}

\begin{proof} {Sufficiency.} Assume that under \(\mathbf{P}\), \(X_t\) is a Markov process with generator \(q_t\), and under \(\mathbf{P}_0\), it is a stationary Markov process with generator \(q_0\) and the same initial distribution. Since \(q_0(i,j) > 0\) for all \(i \neq j\), the process \(X_t\) has positive jump rates under \(\mathbf{P}_0\), and the absolute continuity \(\mathbf{P} \ll \mathbf{P}_0\) follows from the fact that the finite-dimensional distributions under \(\mathbf{P}\) are absolutely continuous with respect to those under \(\mathbf{P}_0\). The likelihood ratio is given by (C).

Define the process
\[
Z_t = \exp\left( \int_0^t \sum_{j \neq X_s} \log \frac{q_s(X_s, j)}{q_0(X_s, j)} \, dN_s^j - \int_0^t \sum_{j \neq X_s} \left( q_s(X_s, j) - q_0(X_s, j) \right) ds \right).
\]
We show that \(Z_t\) is a \(\mathbf{P}_0\)-martingale. Consider the process
\[
U_t = \int_0^t \sum_{j \neq X_s} \log \frac{q_s(X_s, j)}{q_0(X_s, j)} \, (dN_s^j - q_0(X_s, j) ds).
\]
Then \(U_t\) is a local martingale under \(\mathbf{P}_0\) because \(dN_s^j - q_0(X_s, j) ds\) is a martingale increment. Moreover, since \(q_0(i,j) > 0\) and we assume boundedness of the generators, \(U_t\) satisfies the conditions for the stochastic exponential to be a martingale. Thus, the Dol\'eans-Dade exponential \(Z_t = \mathcal{E}(U)_t\) is a \(\mathbf{P}_0\)-martingale.

Now define a probability measure \(\mathbf{P}'\) on $\sigma(X_s:0 \le s \le t)$ by \(d\mathbf{P}' = Z_t \, d\mathbf{P}_0\). By Girsanov's theorem for jump processes (\cite{Sokol03092015}), under \(\mathbf{P}'\), the jump rate of \(X_t\) from state \(i\) to \(j\) becomes \(q_t(i,j)\). Since both \(\mathbf{P}\) and \(\mathbf{P}'\) are Markov measures with generator \(q_t\) and the same initial distribution, they coincide on $\sigma(X_s:0 \le s \le t)$ and moreover \(\mathbf{P} = \mathbf{P}'\) on the full algebra $\sigma(X_s: s\ge 0)$. This is a direct corollary of the following lemma.

\begin{lemma}\label{lem:continuous-generator-vector}
Let \(X_t\) be a continuous-time Markov process with respect to the filtration \(\sigma(X_s:0 \le s \le t)\) under \(\mathbf{P}\) with generator \(q_t\), and under \(\mathbf{P}_0\) it is a stationary Markov process with generator \(q_0\). Suppose \(\mathbf{P} \ll \mathbf{P}_0\). Then there exists a predictable matrix process \(K_t = (K_t(i,j))_{i,j =1}……N\) such that
\[
q_t = q_0 \odot (1 + K_t),
\]
where \(\odot\) denotes the Hadamard product (element-wise multiplication), and \(K_t\) satisfies that for each state \(i\),
\[
\sum_{j=1}^N q_0(i,j) K_t(i,j) = 0 \quad \text{on} \quad \{Z_{t-} > 0\},
\]
where \(Z_t = \left. {d\mathbf{P}}/{d\mathbf{P}_0} \right|{\sigma(X_s:0 \le s \le t)}\) is the likelihood ratio process.
\end{lemma}

\begin{proof}
    By the predictable representation theorem for jump processes under \(\mathbf{P}_0\) (see \cite{Boel1975}, \cite{davis1976representation}, \cite{elliott1976stochastic}), there exist predictable processes \(H_t^j\) for each \(j \in E\) such that
\[
Z_t = 1 + \int_0^t \sum_{j \neq X_s} H_s^j \, dM_s^j,
\]
where \(M_t^j = N_t^j - \int_0^t q_0(X_s, j) \, ds\) are \(\mathbf{P}_0\)-martingales.

From Girsanov's theorem for jump processes (see \cite{Sokol03092015}), the stochastic differential of \(Z_t\) satisfies
\[
dZ_t = Z_{t-} \sum_{j \neq X_{t-}} \left( \frac{q_t(X_{t-}, j)}{q_0(X_{t-}, j)} - 1 \right) dM_t^j.
\]
Comparing these two expressions, we obtain that for each \(j\),
\[
H_t^j = Z_{t-} \left( \frac{q_t(X_{t-}, j)}{q_0(X_{t-}, j)} - 1 \right) \quad \mathbf{P}_0\text{-a.s.}
\]
Now, define the matrix process \(K_t\) by setting for each \(i, j \in E\) with \(i \neq j\)
\[
K_t(i,j) = \frac{H_t^j}{Z_{t-}} \quad \text{on} \quad \{X_{t-} = i, Z_{t-} > 0\},
\]
and for \(i = j\), define \(K_t(i,i) = 0\) initially. Then for \(i \neq j\), we have
\[
q_t(i,j) = q_0(i,j) \left(1 + K_t(i,j)\right).
\]
For the diagonal entries, since generators have row sum zero, we compute
\[
q_t(i,i) = -\sum_{j \neq i} q_t(i,j) = -\sum_{j \neq i} q_0(i,j) (1 + K_t(i,j)) = q_0(i,i) - \sum_{j \neq i} q_0(i,j) K_t(i,j),
\]
where we used \(q_0(i,i) = -\sum_{j \neq i} q_0(i,j)\). Thus, to express \(q_t(i,i)\) in the Hadamard product form, we can define
\[
K_t(i,i) = \frac{ -\sum_{j \neq i} q_0(i,j) K_t(i,j) }{q_0(i,i) } \quad \text{on} \quad \{Z_{t-} > 0\},
\]
which is well-defined since \(q_0(i,i) < 0\). This ensures that
\[
q_t(i,i) = q_0(i,i) (1 + K_t(i,i)).
\]
Therefore, for all \(i, j\), we have
\[
q_t = q_0 \odot (1 + K_t).
\]
The predictability of \(K_t\) follows from the predictability of \(H_t^j\) and \(Z_{t-}\).

Finally, to verify the condition on \(K_t\), note that since both \(q_t\) and \(q_0\) have row sum zero,
\[
\sum_{j=1}^N q_t(i,j) = 0 = \sum_{j=1}^N q_0(i,j) (1 + K_t(i,j)) = \sum_{j=1}^N q_0(i,j) + \sum_{j=1}^N q_0(i,j) K_t(i,j).
\]
Since \(\sum_{j=1}^N q_0(i,j) = 0\), it follows that
\[
\sum_{j=1}^N q_0(i,j) K_t(i,j) = 0,
\]
as required. This completes the proof.
\end{proof}

\begin{lemma}\label{lem: Dynkin lemma} Let \((X_t, q_t)\) be an adapted process, where \(q_t\) is a $\sigma(X_s:0 \le s \le t)$-predictable generator matrix. Then for every function \(f\) on the state space, the process
\begin{equation}
    M_t^f = f(X_t) - f(X_0) - \int_0^t (q_s f)(X_s) ds
    \label{eq: markov martingale continuous case}
\end{equation}is an \(\mathscr{F}_t\)-martingale if and only if \(X_t\) is a Markov process with \(\mathscr{F}_t\)-predictable generator \(q_t\).\end{lemma}
\begin{proof}
If \(X_t\) is Markov with generator \(Q_t\), then by Dynkin's formula, \(M_t^f\) is a martingale. Conversely, if \(M_t^f\) is a martingale for all \(f\), then for any time \(t\) and state \(j\), choose \(f\) to be the indicator function \(f(x) = \mathbf{1}_{\{x = j\}}\). Then
\[
\mathbf{E}[f(X_{t+h}) - f(X_t) \mid \mathscr{F}_t] = \mathbf{E}[\mathbf{1}_{\{X_{t+h} = j\}} \mid \mathscr{F}_t] = \int_t^{t+h} \mathbf{E}[(q_s f)(X_s) \mid \mathscr{F}_t] ds.
\]
Dividing by \(h\) and taking \(h \to 0\), we obtain
\[
\lim_{h \to 0} \frac{1}{h} \mathbf{P}(X_{t+h} = j \mid \mathscr{F}_t) = q_t(X_t, j),
\]
which implies the Markov property with generator \(q_t\).
\end{proof}

{Necessity.} Assume that \(\mathbf{P} \ll \mathbf{P}_0\) and that \(X_t\) is a Markov process with respect to \(\mathscr{F}_t = \sigma(X_s: 0 \le s \le t)\) under \(\mathbf{P}\). By Lemma \ref{lem:continuous-generator-vector} and uniqueness of the martingale problem, the likelihood ratio process must take the form (C).

Now, to confirm that \(X_t\) is indeed a Markov process with generator \(q_t\) under \(\mathbf{P}\), Since we have explicitly constructed \(q_t\) from the Girsanov transformation and the likelihood ratio \(Z_t\) is a \(\mathbf{P}_0\)-martingala, by Lemma \ref{lem: Dynkin lemma}, the process \(M_t^f\) can be shown to be a \(\mathbf{P}\)-martingale by applying Girsanov's theorem so that the compensated process under \(\mathbf{P}_0\) becomes a martingale under \(\mathbf{P}\). This verifies the Markov property with generator \(q_t\). Therefore, the proof is completed.
\end{proof}

\oneappendix 

\begin{acknowledgements}\label{ackref}
In memory of my grandma Lei who passed away at COVID-19.
\end{acknowledgements}

\bibliographystyle{abbrvnat}
\bibliography{bibtex}

\begin{thebibliography}{13}
\providecommand{\natexlab}[1]{#1}
\providecommand{\url}[1]{\texttt{#1}}
\expandafter\ifx\csname urlstyle\endcsname\relax
  \providecommand{\doi}[1]{doi: #1}\else
  \providecommand{\doi}{doi: \begingroup \urlstyle{rm}\Url}\fi

\bibitem[Boel et~al.(1975)Boel, Varaiya, and Wong]{Boel1975}
R.~Boel, P.~Varaiya, and E.~Wong.
\newblock Martingales on jump processes. {I}: Representation results. {II}: Applications.
\newblock \emph{SIAM J. Control}, 13\penalty0 (5):\penalty0 999--1021, 1022–1061, 1975.

\bibitem[Buiculescu(1982)]{buiculescu1982representation}
M.~Buiculescu.
\newblock Representation of excessive functions as potentials of additive functionals for regular {M}arkov processes.
\newblock \emph{Theor. Prob. Appl.}, 26\penalty0 (2):\penalty0 385--388, 1982.

\bibitem[Clark(1970)]{clark1970representation}
J.~M.~C. Clark.
\newblock The representation of functionals of {B}rownian motion by stochastic integrals.
\newblock \emph{Ann. Math. Stat.}, pages 1282--1295, 1970.

\bibitem[Criens and Urusov(2024)]{Criens2024OnTR}
D.~Criens and M.~Urusov.
\newblock On the representation property for 1d general diffusion semimartingales.
\newblock \emph{Theor. Prob. Appl.}, 69:\penalty0 729–744, 2024.

\bibitem[Davis(1976)]{davis1976representation}
M.~H.~A. Davis.
\newblock The representation of martingales of jump processes.
\newblock \emph{SIAM J. Control Optim.}, 14\penalty0 (4):\penalty0 623--638, 1976.

\bibitem[Dol\'eans-Dade(1970)]{DoleansDade1970}
C.~Dol\'eans-Dade.
\newblock Quelques applications de la formule de changement de variables pour les semimartingales.
\newblock \emph{Prob. Relat. Field}, 16\penalty0 (3):\penalty0 181--194, 1970.

\bibitem[Dynkin(1965)]{dynkin1965markov}
E.~Dynkin.
\newblock \emph{Markov processes, {I-II}}.
\newblock Springer, 1965.

\bibitem[Elliott(1976)]{elliott1976stochastic}
R.~J. Elliott.
\newblock Stochastic integrals for martingales of a jump process with partially accessible jump times.
\newblock \emph{Prob. Relat. Field}, 36\penalty0 (3):\penalty0 213--226, 1976.

\bibitem[Girsanov(1960)]{girsanov1960}
I.~V. Girsanov.
\newblock On transforming a certain class of stochastic processes by absolutely continuous substitution of measures.
\newblock \emph{Theor. Prob. Appl.}, 5\penalty0 (3):\penalty0 285--301, 1960.

\bibitem[Kazamaki(1994)]{kazamaki1994continuous}
N.~Kazamaki.
\newblock \emph{Continuous Exponential Martingales and {BMO}}, volume 1579.
\newblock Springer-Verlag, 1994.

\bibitem[Novikov(1972)]{novikov1972identity}
A.~A. Novikov.
\newblock On an identity for stochastic integrals.
\newblock \emph{Theor. Prob. Appl.}, 17\penalty0 (4):\penalty0 761--765, 1972.

\bibitem[Sokol and Hansen(2015)]{Sokol03092015}
A.~Sokol and N.~R. Hansen.
\newblock Exponential martingales and changes of measure for counting processes.
\newblock \emph{Stoch. Anal. Appl.}, 33\penalty0 (5):\penalty0 823--843, 2015.

\bibitem[Stroock and Varadhan(1969)]{StroockVaradhan1969}
D.~W. Stroock and S.~R.~S. Varadhan.
\newblock Diffusion processes with continuous coefficients. {I-II}.
\newblock \emph{Comm. Pure Appl. Math.}, 22:\penalty0 345--400, 479--530, 1969.

\end{thebibliography}

\affiliationone{Haoming Wang\\
   School of Mathematics \\
   Sun Yat-sen University \\
   Guangzhou, China
   \email{wanghm37@mail2.sysu.edu.cn}}

\end{document}